\documentclass[reqno]{amsart}
\usepackage{amssymb}
\usepackage{amsmath}
\usepackage{amsfonts}
\usepackage{graphicx}
\usepackage{subfig}
\usepackage{eurosym}
\usepackage{amssymb}
\usepackage{amsmath}
\usepackage{amsfonts}
\usepackage{graphicx}
\usepackage{subfig}

\newtheorem{theorem}{Theorem}
\theoremstyle{plain}

\newtheorem{definition}{Definition}

\numberwithin{equation}{section}

\begin{document}
\title[Batman University]{A Generalization of $g$-rectifying curves and $g-$%
normal curves in Lorentzian $n$-Space}
\author{Fatma ALMAZ}
\address{department of mathematics, faculty of arts and sciences, batman
university, batman/ t\"{u}rk\.{ı}ye orcıd: 0000-0002-1060-7813}
\email{fatma.almaz@batman.edu.tr}
\author{Hazel D\.{I}KEN}
\address{department of mathematics, faculty of arts and sciences, batman
university, batman/ t\"{u}rk\.{ı}ye orcıd: 0009-0008-9038-5624}
\email{hazeldiken@hotmail.com}
\thanks{This paper is in final form and no version of it will be submitted
for publication elsewhere.}

\begin{abstract}
In this paper, we introduce and analyze $g-$rectifying curves
(spacelike and null curves) and $\ g-$normal curves in Lorentzian $n$-space,
building upon the established notion of rectifying curves and normal curve,
respectively. Our generalization extends this definition by considering an $%
g-$position vector field, $\xi _{g}(s)=\int g(s)d\xi $, where $g$ is an
integrable function in the arc-length parameter $s$. An $g$-rectifying
curves(or $g-$normal curves) are then defined as an arc-length parametrized
curve $\xi $ in Lorentzian $n-$space such that its $g$-position vector
consistently lies within its rectifying space(or normal space). The primary
objective of this work is to provide a comprehensive characterization and
classification of these $g$-rectifying curves and $g-$normal curves, thereby
expanding the geometric understanding of curves in Lorentzian $n$-spaces.
\end{abstract}

\maketitle

\section{Introduction}

The study of curves within differential geometry forms a cornerstone of
modern mathematics, offering profound insights into the intrinsic properties
of geometric objects. Particularly, investigations in Lorentzian $n-$spaces
have gained significant traction due to their fundamental role in physics,
notably in the context of general relativity, where space-time is modelled as
a Lorentzian manifold. Understanding the behaviour and characteristics of
curves in such indefinite metric spaces is crucial for both theoretical
advancements and potential applications.

Traditionally, concepts such as rectifying curves and normal curves have
been extensively studied, providing elegant geometric classifications based
on the relationship between a curve's position vector and its associated
Frenet frames. A rectifying curve is classically defined as a curve whose
position vector always lies in its rectifying space, while a normal curve's
position vector is confined to its normal space. These definitions have
proven instrumental in characterizing the geometry and kinematics of various
curves.

In this paper, we extend these established notions by introducing a novel
generalization: the $g-$position vector field. For an arc-length
parametrized curve $\xi $, we define its $g-$position vector field as $\xi
_{g}(s)=\int g(s)d\xi $, where $g$ is a nowhere vanishing integrable
function dependent on the arc-length parameter $s$. This new construct
allows for a more flexible and comprehensive analysis of curve geometries,
moving beyond the standard position vector.

Building upon this generalization, we formally define $g-$rectifying curves
and $g-$normal curves in Lorentzian $n-$space. An arc-length parametrized
curve $\xi $ is termed an $g-$rectifying curve if its $g-$position vector
consistently resides within its rectifying space. Analogously, it is
designated an $g-$normal curve if its $g-$position vector remains confined
to its normal space. Our analysis specifically considers both spacelike and
null curve scenarios, acknowledging the unique properties conferred by the
Lorentzian metric.

The primary objective of this work is to provide a comprehensive
characterization and classification of these newly introduced $g-$rectifying
and $g-$normal curves. Through rigorous mathematical derivations and
geometric interpretations, we aim to elucidate the conditions under which
these curves exist and to explore their fundamental properties. By expanding
the conceptual framework of curve theory, this research endeavors to deepen
the geometric understanding of curves in Lorentzian $n-$spaces, potentially
paving the way for new avenues of inquiry in both pure and applied
differential geometry.

In references \cite{1} and \cite{2}, the authors provided various
mathematical characterizations for rectifying and osculating curves in the
lightlike cone space. In \cite{3}, this study generalizes normal curves in $%
n $-dimensional Euclidean space, providing necessary and sufficient
conditions for a curve to be normal, characterizing the equivalence of
unit-speed curves to normal curves using curvature relationships, and
solving the differential equation defining these curves. The concept of a
rectifying curve in $E^{3}$ was initially established by Chen \cite{4}.
Subsequent to this foundational work, the realm of differential geometry has
witnessed a proliferation of diverse perspectives concerning rectifying,
normal, and osculating curves. Building upon their work, In \cite{5}, the
authors proposed that Euclidean rectifying curves constitute extremal
solutions adhering to the equality condition of a specific inequality.
Furthermore, they established a straightforward correlation linking
rectifying curves to the concept of centrodes, a fundamental notion in
mechanics. The properties of rectifying curves in n-dimensional Euclidean
geometry have been a subject of investigation by the authors in \cite{6}.
The inherent relationships between rectifying and normal curves in Minkowski
3-space were formally established in \cite{7}. Complementing this, the
characteristics of spacelike, timelike, and null normal curves within
Minkowski space have been thoroughly explored in \cite{9}. Detailed
characterizations of osculating, normal, and rectifying binull curves in $%
R_{2}^{5}$ have been presented in \cite{10}. In \cite{11}, this paper aims
to characterize and classify f-rectifying curves in Euclidean $n-$space by
defining them as curves whose f-position vector field always lies within
their rectifying space. Concurrently, the definitions of normal and
rectifying curves were extended to Galilean space in \cite{13}.
Complementarily, the definition of normal curves utilizing quaternions in
Euclidean space has been established, and their investigation within
semi-Euclidean spaces was notably conducted by the authors in \cite{16}.

\section{Preliminaries}

The Lorentzian $n-$dimensional space $L^{n}$ is the standard vector space $%
R^{n}$ endowed with the geometrical structure given by the Lorentzian scalar
product 
\begin{equation}
\ g(x,y)=-x_{1}y_{1}+x_{2}y_{2}+...+x_{n-1}y_{n-1}+x_{n}y_{n},  \tag{2.1}
\end{equation}%
$\forall $ $x=(x_{1},...,x_{n-1},x_{n}),$ $y=(y_{1},...,y_{n-1},y_{n})$ in $%
R^{n}$. A vector $W=(w_{1},...,w_{n-1},w_{n})$ in $R^{n}$ is called
spacelike, timelike or null (lightlike) when respectively $g(W,W)>0,g(W,W)<0$
or $g(W,W)=0$ and $W=0,$a non null vector $W$ is said to be future pointing
or past pointing when respectively $g(W,E)<0$ or $g(W,E)>0$ where by $%
E=(0,0,...,1)$, i.e. when $w_{n}>0$ or $w_{n}<0,$ $\left\Vert W\right\Vert =%
\sqrt{g(W,W)}$ is called the norm of length of $W$, and two vectors $V$ and $%
W$ in $L^{n}$ are said to be orthogonal when $g(V,W)=0,$ \cite{14}.

A curve $\xi $ in $L^{n}$ is said to be spacelike if all of its velocity
vectors $\xi $ are spacelike, it is similar for timelike and null, \cite%
{14,15}. Let $\xi :I\subset 
\mathbb{R}
\longrightarrow L^{n}$ be an arc length parametrized spacelike curve. Let $%
T=\xi ^{\prime }$ and $N$ denote the unit tangent vector field and unit
principal normal vector field of $\xi $ for each $i\in \left\{
1,2,...,n-2\right\} ,$ let $B_{i}$ denote $i-$th bi-normal vector field of $%
\xi $ so that $\{T,N,B_{1},...,B_{n-2}\}$ forms the positive definite Frenet
frame along $\xi ,$ let $\kappa _{1},\kappa _{2},...,\kappa _{i},...,\kappa
_{n-1}$ are the curvatures and $\varepsilon _{1},\varepsilon
_{2},...,\varepsilon _{i},...,\varepsilon _{n-1}$ are the signatures of $%
\{N,B_{1},...,B_{n-2}\}.$ Therefore, $\{T,N,B_{1},...,B_{n-2}\}$ is an
orthonormal frame of $\xi $ and the Frenet equations of the spacelike curve $%
\xi $ are as follows 
\begin{eqnarray*}
T^{\prime } &=&\nabla _{T}\xi ^{\prime }=\kappa _{1}N \\
N^{\prime } &=&\nabla _{T}N=-\varepsilon _{1}\kappa _{1}T+\kappa _{2}B_{1} \\
B_{1}^{\prime } &=&\nabla _{T}B_{1}=-\varepsilon _{1}\varepsilon _{2}\kappa
_{2}N+\kappa _{3}B_{2},
\end{eqnarray*}%
\begin{equation*}
...
\end{equation*}%
\begin{equation}
B_{i}^{\prime }=\nabla _{T}B_{i}=-\varepsilon _{i}\varepsilon _{i+1}\kappa
_{i+1}B_{i-1}+\kappa _{i+2}B_{i+1},  \tag{2.2}
\end{equation}%
\begin{equation*}
...
\end{equation*}%
\begin{equation*}
B_{n-2}^{\prime }=\nabla _{T}B_{n-2}=-\varepsilon _{n-2}\varepsilon
_{n-1}\kappa _{n-1}B_{n-3},
\end{equation*}%
where $\nabla $ is the Levi Civita connection of $L^{n}$ and $\varepsilon
_{1}=\left\langle N,N\right\rangle ,$ $\varepsilon _{2}=\left\langle
B_{1},B_{1}\right\rangle ,$ $...$ $,\varepsilon _{n-1}=\left\langle
B_{n-2},B_{n-2}\right\rangle ,$ \cite{8,12,14,15}.

For fundamental insights into the theory of curves within differential
geometry, readers are directed to the comprehensive resources provided in
references \cite{12}, \cite{14}, and \cite{15}.

\begin{definition}
(\cite{8})Let $\xi :I\subset 
\mathbb{R}
\longrightarrow E^{n}$ is called rectifying curve if for all $s\in I$. Then,
for $N^{\bot }$ the orthogonal complement of $N$, the position vector of a
spacelike rectifying curve $\xi $ in $E^{n}$ can be written as%
\begin{equation}
\xi ^{r}(s)=w_{0}^{r}T(s)+\overset{n-2}{\underset{i=1}{\sum }}%
w_{i}^{r}B_{i}(s);w_{0},w_{1},...,w_{n-2}\in C^{\infty }.  \tag{2.3}
\end{equation}
\end{definition}

\begin{definition}
(\cite{8})Let $\xi :I\subset 
\mathbb{R}
\longrightarrow E^{n}$ is called normal curve for $\forall $ $s\in I$. Then,
for $T^{\bot }$ the orthogonal complement of $T$, the position vector of a
spacelike rectifying curve $\xi $ in $E^{n}$ can be written as%
\begin{equation}
\xi ^{n}(s)=w_{0}^{n}N(s)+\overset{n-2}{\underset{i=1}{\sum }}%
w_{i}^{n}B_{i}(s);w_{0},w_{1},...,w_{n-2}\in C^{\infty }.  \tag{2.4}
\end{equation}
\end{definition}

Let $g:I\longrightarrow 
\mathbb{R}
$ be a integrable function. Then, the $g$-position vector field of $\xi $ is
denoted by $\xi _{g}$ and is defined as 
\begin{equation}
\xi _{g}(s)=\int g(s)d\xi  \tag{2.5}
\end{equation}%
that is, on differentiation of previous equation, one gets%
\begin{equation}
\xi _{g}^{\prime }(s)=g(s)T_{\xi }(s).  \tag{2.6}
\end{equation}

Then, $\xi _{g}$ is called an integral curve of the vector field $gT_{\xi }$
along $\xi $ in $E^{n}.$

\section{Representation of $g-$rectifying curves in Lorentzian $n-$space $%
L^{n}$}

In this section, some characterizations of unit-speed spacelike $g-$%
rectifying curves (or null curves) in $L^{n}$ are expressed and given
tangential component, normal component, binormal components of their $g-$%
position vector field.

\begin{definition}
Let $\xi :I\subset 
\mathbb{R}
\longrightarrow L^{n}$ be a unit-speed spacelike curve with Frenet frame $%
\left\{ T_{\xi },N_{\xi },B_{1}^{\xi },B_{2}^{\xi },...,B_{n-2}^{\xi
}\right\} $ and $g:I\longrightarrow 
\mathbb{R}
$ be a no where vanishing integrable function in arc-length parameters of $%
\xi $. Then, $\xi _{g}^{r}$ is referred to as a spacelike $g-$rectifying
curve in $L^{n}$ if its $g-$position vector field $\xi _{g}^{r}$ always lies
in its rectifying space in $L^{n}$, if its $g-$position vector of a
spacelike $g-$rectifying curve $\xi _{g}^{r}$ is given as 
\begin{equation}
\xi _{g}^{r}(s)=w_{0}^{r}T_{\xi }(s)+\overset{n-2}{\underset{i=1}{\sum }}%
w_{i}^{r}B_{i}^{\xi }(s),  \tag{3.1}
\end{equation}%
for differentiable functions $w_{i}^{r}:I\longrightarrow 
\mathbb{R}
;i=0,1,...,n-2.$
\end{definition}

\begin{theorem}
Let $\xi :I\subset 
\mathbb{R}
\longrightarrow L^{n}$ be a unit-speed spacelike curve having no where
vanishing $n-1$ curvatures $\kappa _{1},\kappa _{2},...,\kappa _{n-1},$ and
let $g:I\rightarrow R$ be a no where vanishing integrable function with at
least $(n-2)$-times differentiable primitive function $G$. If $\xi $ is a
spacelike $g-$rectifying curve in $L^{n}$, then the following statements are
true:

1) The norm function $l$ associated with the $g-$position vector field $\xi
_{g}^{r}$ is explicitly defined as $l^{2}=-G(s)^{2}+c^{2}$.

2) The tangential projection of the $g-$position vector field $\xi _{g}^{r}$
onto the tangent vector $T_{\xi }$ is given by the scalar product $%
\left\langle \xi _{g}^{r},T_{\xi }\right\rangle =G(s).$

3) The normal component $\xi _{g}^{rN}(s)$ of the $g-$position vector field $%
\xi _{g}^{r}$ maintains a constant magnitude.

4) The binormal components of the $g-$position vector field $\xi _{g}^{r}$
are, respectively, provided by the following expressions%
\begin{equation*}
\left\langle \xi _{g}^{r},B_{1}^{\xi }\right\rangle =\frac{\kappa _{1}}{%
\varepsilon _{1}\kappa _{2}}G;\left\langle \xi _{g}^{r},B_{2}^{\xi
}\right\rangle =\frac{1}{\varepsilon _{1}\kappa _{3}}(\frac{\kappa _{1}}{%
\kappa _{2}}G)^{\prime },
\end{equation*}%
\begin{equation*}
\left\langle \xi _{g}^{r},B_{i+1}^{\xi }\right\rangle =\frac{1}{\varepsilon
_{i+1}\kappa _{i+2}}(\mu _{i}^{\prime }+\kappa _{i+1}\mu
_{i-1});\left\langle \xi _{g}^{r},B_{n-2}^{\xi }\right\rangle =-\varepsilon
_{n-1}\int \kappa _{n-1}w_{n-3}^{r}ds,
\end{equation*}

5) The spacelike $g-$rectifying curve of $\xi $ is given as 
\begin{equation*}
\xi _{g}^{r}(s)=\left( 
\begin{array}{c}
G(s),0,\frac{\kappa _{1}}{\varepsilon _{1}\kappa _{2}}G,\frac{1}{\varepsilon
_{1}\kappa _{3}}(\frac{\kappa _{1}}{\kappa _{2}}G)^{\prime }, \\ 
...,\frac{1}{\varepsilon _{i+1}\kappa _{i+2}}(\mu _{i}^{\prime }+\kappa
_{i+1}\mu _{i-1}),...,-\varepsilon _{n-1}\int \kappa _{n-1}w_{n-3}^{r}ds%
\end{array}%
\right) ,
\end{equation*}%
where $G(s)$ represents the primitive function and $c$ is a specified
non-zero constant, $i=2,3,...,n-3$.
\end{theorem}

\begin{proof}
Consider a nowhere vanishing integrable function $g:I\rightarrow R$,
possessing an $(n-2)-$times differentiable primitive function $G$. Let $%
\gamma :I\subset 
\mathbb{R}
\longrightarrow L^{n}$ be a spacelike $g-$rectifying curve in $L^{n}$,
characterized by its nowhere vanishing $(n-1)$ curvatures $\kappa
_{1},\kappa _{2},...,\kappa _{n-1}$. Under these premises, there exist
differentiable functions $w_{i}^{r}\in C^{\infty }$(for $i=0,1,2,...,n-2$ )
such that the $g-$position vector field $\xi _{g}^{r}$ of $\xi $ satisfies
equation (3.1). Then, upon differentiation of (4.1) and subsequent
application of the Frenet-Serret formulae (2.2), one obtains 
\begin{equation*}
g(s)\overrightarrow{T}=w_{0}^{r\prime }\overrightarrow{T}_{\xi }+(\kappa
_{1}w_{0}^{r}-\varepsilon _{1}\varepsilon _{2}\kappa _{2}w_{1}^{r})%
\overrightarrow{N}_{\xi }+(w_{1}^{r\prime }-\varepsilon _{2}\varepsilon
_{3}\kappa _{3}w_{2}^{r})\overrightarrow{B_{1}^{\xi }}
\end{equation*}%
\begin{equation}
+\underset{i=2}{\overset{n-3}{\sum }}(w_{i}^{r\prime }+\kappa
_{i+1}w_{i+1}^{r}-\varepsilon _{i+1}\varepsilon _{i+2}\kappa
_{i+2}w_{i+1}^{r})\overrightarrow{B_{i}^{\xi }}+(w_{n-2}^{r\prime }+\kappa
_{n-1}w_{n-3}^{r})\overrightarrow{B_{n-2}^{\xi }}.  \tag{3.2}
\end{equation}

The resulting set of relations is as follows:%
\begin{equation}
g(s)=w_{0}^{r}{}^{\prime }  \tag{3.3}
\end{equation}%
\begin{equation}
\kappa _{1}w_{0}^{r}-\varepsilon _{1}\varepsilon _{2}\kappa _{2}w_{1}^{r}=0 
\tag{3.4a}
\end{equation}%
\begin{equation}
w_{1}^{r\prime }-\varepsilon _{2}\varepsilon _{3}\kappa _{3}w_{2}^{r}=0 
\tag{3.4b}
\end{equation}%
\begin{equation}
w_{i}^{r\prime }+\kappa _{i+1}w_{i+1}^{r}-\varepsilon _{i+1}\varepsilon
_{i+2}\kappa _{i+2}w_{i+1}^{r}=0  \tag{3.4c}
\end{equation}%
\begin{equation}
w_{n-2}^{r\prime }+\kappa _{n-1}w_{n-3}^{r}=0.  \tag{3.4d}
\end{equation}

Utilizing the ($n-1$) relations within the previously mentioned system of
equations, one obtains the following equalities%
\begin{equation}
w_{0}^{r}=\int g(s)ds=G(s)  \tag{3.5a}
\end{equation}%
\begin{equation}
w_{1}^{r}=\frac{\kappa _{1}}{\varepsilon _{1}\varepsilon _{2}\kappa _{2}}G 
\tag{3.5b}
\end{equation}%
\begin{equation}
w_{2}^{r}=\frac{1}{\varepsilon _{1}\varepsilon _{3}\kappa _{3}}(\frac{\kappa
_{1}}{\kappa _{2}}G)^{\prime }  \tag{3.5c}
\end{equation}%
\begin{equation}
w_{n-2}^{r}=-\int \kappa _{n-1}w_{n-3}^{r}ds  \tag{3.5d}
\end{equation}%
\begin{equation}
w_{i+1}^{r}=\frac{1}{\varepsilon _{i+1}\varepsilon _{i+2}\kappa _{i+2}}%
(w_{i}^{r\prime }+\kappa _{i+1}w_{i-1}^{r}),i=2,3,...,n-3.  \tag{3.5e}
\end{equation}

Finally, upon multiplying equations (3.4b), (3.4c), and (3.4d) by $w_{1}^{r}$%
, $w_{i}^{r}$, and $w_{n-2}^{r}$, for $i=2,3,...,n-3$, respectively, and
summing them, the following equality is obtained. 
\begin{equation*}
\varepsilon _{2}w_{1}^{r}w_{1}^{r\prime }-\varepsilon _{3}\kappa
_{3}w_{2}^{r}w_{1}^{r}+\underset{i=2}{\overset{n-3}{\sum }}(\varepsilon
_{i+1}w_{i}^{r}w_{i}^{r\prime }+\varepsilon _{i+1}\kappa
_{i+1}w_{i-1}^{r}w_{i}^{r}-\varepsilon _{i+1}\varepsilon _{i+2}\kappa
_{i+2}w_{i+1}^{r}w_{i}^{r})
\end{equation*}%
\begin{equation*}
+w_{n-2}^{r}w_{n-2}^{r\prime }+\kappa _{n-1}w_{n-3}^{r}w_{n-2}^{r}=0.
\end{equation*}

Ultimately, by performing the necessary calculations, the following
expression is obtained.%
\begin{equation}
\underset{i=2}{\overset{n-3}{\sum }}\varepsilon
_{i+1}w_{i}^{r}w_{i}^{r\prime }=0\Rightarrow \underset{i=2}{\overset{n-3}{%
\sum }}\varepsilon _{i+1}w_{i}^{r}{}^{2}=c^{2},  \tag{3.6}
\end{equation}%
for some arbitrary non-zero constant $c$. Also, from equations (3.1), (3.5),
and (3.6), the norm of the curve is calculated as follows%
\begin{equation}
l^{2}=\left\Vert \xi _{g}^{r}\right\Vert ^{2}=-w_{0}^{r2}+\underset{i=2}{%
\overset{n-3}{\sum }}\varepsilon
_{i+1}w_{i}^{r}{}^{2}=-w_{0}^{r}{}^{2}+c^{2}.  \tag{3.7}
\end{equation}

Also, to determine the tangential component for the given $g-$rectifying
curve, the following equality are obtained from equations (3.1) and (3.5),
respectively%
\begin{equation*}
\left\langle \xi _{g},T_{\xi }\right\rangle =w_{0}^{r}=G(s).
\end{equation*}

Also, we can express this as follows using the normal component of curve $%
\xi _{g}^{r}$ 
\begin{equation}
\xi _{g}^{r}(s)=w_{0}^{r}T_{\xi }(s)+\overset{n-2}{\underset{i=1}{\sum }}%
w_{i}^{r}B_{i}^{\xi }(s)=w_{0}^{r}T_{\xi }(s)+\xi _{f}^{rN}(s).  \tag{3.8}
\end{equation}

Consequently, the norm of the normal component is computed as follows%
\begin{equation}
\left\Vert \xi _{g}^{rN}(s)\right\Vert ^{2}=\sqrt{\overset{n-2}{\underset{i=1%
}{\sum }}w_{i}^{r}{}^{2}\varepsilon _{i+1}}=c,  \tag{3.9}
\end{equation}%
the norm of the normal component is found to be constant, thereby completing
the proof of (3). Now, considering equations (3.1) and (3.5), the binormal
components are obtained as follows, respectively%
\begin{equation}
\left\langle \xi _{g}^{r},B_{i}^{\xi }\right\rangle =w_{i}^{r}\varepsilon
_{i+1}  \tag{3.10}
\end{equation}%
for $i=1,2,n-2,$ one gets%
\begin{equation}
\left\langle \xi _{g}^{r},B_{1}^{\xi }\right\rangle =w_{1}^{r}\varepsilon
_{2}=\frac{\kappa _{1}}{\varepsilon _{1}\kappa _{2}}G;\left\langle \xi
_{g}^{r},B_{2}^{\xi }\right\rangle =\frac{1}{\varepsilon _{1}\kappa _{3}}(%
\frac{\kappa _{1}}{\kappa _{2}}G)^{\prime }  \tag{3.11}
\end{equation}%
\begin{equation}
\left\langle \xi _{g}^{r},B_{n-2}^{\xi }\right\rangle =-\varepsilon
_{n-1}\int \kappa _{n-1}w_{n-3}^{r}ds  \tag{3.12}
\end{equation}%
and for $i=2,...,n-3,$ 
\begin{equation}
\left\langle \xi _{g}^{r},B_{i+1}^{\xi }\right\rangle
=w_{i+1}^{r}\varepsilon _{i+2}=\frac{1}{\varepsilon _{i+1}\kappa _{i+2}}%
(w_{i}^{r\prime }+\kappa _{i+1}w_{i-1}^{r})  \tag{3.13}
\end{equation}%
hence, the proof of statement (4) is successfully concluded.

Conversely, consider a unit-speed $g-$rectifying curve $\xi :I\subset 
\mathbb{R}
\longrightarrow L^{n}$, characterized by its nowhere vanishing $(n-1)$
curvatures $\kappa _{i},i=1,2,...,n-2$. Additionally, let $g:I\rightarrow R$
be a nowhere vanishing integrable function whose primitive function $G$ is
differentiable at least $(n-2)$ times. If either statement (1) or statement
(2) holds true, then it necessarily follows that 
\begin{equation*}
\left\langle \xi _{g}^{r},T_{\xi }\right\rangle =w_{0}^{r}=G(s),
\end{equation*}%
upon differentiating the last equality and considering equations (2.2), the
following equality is obtained%
\begin{equation*}
\left\langle \xi _{g}^{r},N_{\xi }\right\rangle =0,
\end{equation*}%
this observation leads to the deduction that $\xi _{g}^{r}$ lies within the
rectifying space of $\xi $, thus proving $\xi $ to be an $g-$rectifying
curve in $L^{n}.$

Let us now assume that statement (3) holds true. In this case, the normal
component $\xi _{g}^{rN}$ is a constant, denoted as $c$. This component is
explicitly given by%
\begin{equation*}
\xi _{g}^{r}(s)=w_{0}^{r}T(s)+\xi _{g}^{rN}(s).
\end{equation*}

Thus, the norm is written as follows%
\begin{equation*}
\left\Vert \xi _{g}^{r}\right\Vert ^{2}=-w_{0}^{r}{}^{2}+c^{2}
\end{equation*}%
and by differentiating the previous equation and subsequently applying the
Frenet-Serret formulae (2.2), one obtains $\left\langle \xi _{g},N_{\xi
}\right\rangle =0$. Consequently, it is established that $\xi _{g}^{r}$ lies
within the rectifying space of $\xi $, thereby affirming $\xi $ as an $g-$%
rectifying curve in $L^{n}$. Also, by assuming the truth of statement (4),
the first and second binormal components of $\xi _{g}^{r}$ are provided by%
\begin{equation*}
\left\langle \xi _{g}^{r},B_{1}^{\xi }\right\rangle =\frac{\kappa _{1}}{%
\varepsilon _{1}\kappa _{2}}G;\left\langle \xi _{g}^{r},B_{2}^{\xi
}\right\rangle =\frac{1}{\varepsilon _{1}\kappa _{3}}(\frac{\kappa _{1}}{%
\kappa _{2}}G)^{\prime }
\end{equation*}%
and application of (2.2) to the derivative of the previous equations leads
to the following%
\begin{equation*}
-\varepsilon _{1}\varepsilon _{2}\kappa _{2}\left\langle \xi _{g}^{r},N_{\xi
}\right\rangle +\kappa _{3}\left\langle \xi _{g}^{r},B_{2}^{\xi
}\right\rangle =\left( \frac{\kappa _{1}}{\varepsilon _{1}\kappa _{2}}%
G\right) ^{\prime }
\end{equation*}%
and consequently, $\left\langle \xi _{g}^{r},N_{\xi }\right\rangle =0$ is
obtained, which demonstrates that the curve is a spacelike $g-$rectifying
curve in $L^{n}.$ The proof of statement (5) is straightforward.
\end{proof}

\begin{theorem}
Consider a unit-speed spacelike curve $\xi :I\subset 
\mathbb{R}
\rightarrow L^{n}$, possessing nowhere vanishing $(n-1)$ curvatures $\kappa
_{1},\kappa _{2},...,\kappa _{n-1}$. Let $g:I\rightarrow R$ be a nowhere
vanishing integrable function with an at least $(n-2)$-times differentiable
primitive function $G$. Then, $\xi _{g}^{r}$ constitutes an $g-$rectifying
spacelike curve in $L^{n}$ if and only if, up to re-parametrization, its $g-$%
position vector field $\xi _{g}^{r}$ can be expressed as 
\begin{equation*}
\xi _{g}^{r}(t)=\zeta (t)c\cos (t+\arcsin \frac{G(s_{0})}{c}),
\end{equation*}%
with $c$ being a positive constant and $s_{0}\in I$, and $\zeta
:J\rightarrow H^{n-1}$ is defined as a unit-speed curve, where $%
t:I\rightarrow J$ serves as its arc-length function.
\end{theorem}

\begin{proof}
Consider a integrable function $g$, whose primitive function $G$ is
differentiable at least $(n-2)$ times. Let $\xi _{g}^{r}$ denote an $g-$%
rectifying curve characterized by curvatures $\kappa _{1},\kappa
_{2},...,\kappa _{n-1}$, from (3.7), the norm function is then determined by 
\begin{equation*}
l^{2}=\left\Vert \xi _{g}^{r}\right\Vert ^{2}=-w_{0}^{r}{}^{2}+c^{2}.
\end{equation*}

To proceed, we introduce a curve $\zeta $ defined by the relationship 
\begin{equation}
\zeta (s)=\frac{\xi _{g}^{r}}{l}.  \tag{3.14}
\end{equation}

A direct computation subsequently reveals that 
\begin{equation}
\left\langle \zeta (s),\zeta (s)\right\rangle =-1,  \tag{3.15}
\end{equation}%
which unequivocally establishes $\zeta $ as a curve residing on the unit
hyperbolic space $H^{n-1}(-1)$. Differentiating equation (3.15) then yields
the following result%
\begin{equation*}
\left\langle \zeta ^{\prime }(s),\zeta (s)\right\rangle =0.
\end{equation*}

Finally, by considering the value of $l$ in conjunction with equation
(3.14), the following equation is obtained%
\begin{equation}
\xi _{g}^{r}(s)=\zeta (s)\sqrt{-G^{2}+c^{2}}.  \tag{3.16}
\end{equation}

Subsequently, taking the derivative of (3.16), we find%
\begin{equation}
\xi _{g}^{r\prime }(s)=\zeta ^{\prime }(s)\sqrt{-G^{2}+c^{2}}-\frac{Gg}{%
\sqrt{-G^{2}+c^{2}}}\zeta (s)  \tag{3.17}
\end{equation}%
and from (3.14), we get%
\begin{equation}
\zeta ^{\prime }=\frac{-l^{\prime }\xi _{g}^{r}}{l^{2}}+\frac{g(s)T_{\xi }}{l%
}.  \tag{3.18}
\end{equation}

From the equations (3.15), (3.16), (3.17) and $\left\langle \zeta
^{\prime }(s),\zeta (s)\right\rangle =0$, we deduce that%
\begin{equation*}
\left\langle \xi _{g}^{r\prime }(s),\zeta ^{\prime }(s)\right\rangle
=\left\langle \zeta ^{\prime }(s),\zeta ^{\prime }(s)\right\rangle \sqrt{%
-G^{2}+c^{2}}-\frac{Gg}{\sqrt{-G^{2}+c^{2}}}\left\langle \zeta (s),\zeta
^{\prime }(s)\right\rangle
\end{equation*}%
\begin{equation}
\left\langle g(s)T(s),\zeta ^{\prime }(s)\right\rangle =\left\Vert \zeta
^{\prime }(s)\right\Vert ^{2}\sqrt{-F^{2}+c^{2}}  \tag{3.19}
\end{equation}%
and from (3.19) and (3.17), we get%
\begin{equation}
\frac{g(s)}{\sqrt{-G^{2}(s)+c^{2}}}=\left\Vert \zeta ^{\prime }(s)\right\Vert
\tag{3.20}
\end{equation}%
with $s\in I$, the arc-length parametrization of $\zeta $, represented by $t$%
, is formulated as follows%
\begin{equation*}
t=\int_{s_{0}}^{s}\left\Vert y^{\prime }(x)\right\Vert dx\Longrightarrow
t=\int_{s_{0}}^{s}\frac{g(x)}{\sqrt{-G^{2}(x)+c^{2}}}dx=\int_{s_{0}}^{s}%
\frac{G^{\prime }(x)}{\sqrt{-G^{2}(x)+c^{2}}}dx
\end{equation*}%
\begin{equation}
t=\arcsin \frac{G(s)}{c}-\arcsin \frac{G(s_{0})}{c}\Longrightarrow
G(s)=c\sin \left( t+\arcsin \frac{G(s_{0})}{c}\right)  \tag{3.21}
\end{equation}%
and%
\begin{equation}
s=G^{-1}\left( c\sin \left( t+\arcsin \frac{G(s_{0})}{c}\right) \right) . 
\tag{3.22}
\end{equation}

Finally, substituting the equality from (3.21) into (3.16) yields, we get%
\begin{equation}
\xi _{g}^{r}(t)=\zeta (t)c\cos (t+\arcsin \frac{G(s_{0})}{c}).  \tag{3.23}
\end{equation}
\end{proof}

\subsection{Representation of null $g-$rectifying curves in Lorentzian $n-$%
space $L^{n}$}

In this subsection, some characterizations of unit-speed null $g-$rectifying
curves in $L^{n}$ are presented, elucidated through the analysis of the
tangential, normal, and binormal components of their associated $g-$position
vector fields.

Let $\xi :I\subset 
\mathbb{R}
\longrightarrow L^{n}$ be an arc length parametrized null curve. Let $T=\xi
^{\prime }$ and $N$ denote the unit tangent vector field and unit principal
normal vector field of $\xi $ for each $i\in \left\{ 1,2,...,n-2\right\} ,$
let $B_{i}^{\xi }$ denote $i-$th binormal vector field of $\xi $ so that $%
\{T_{\xi },N_{\xi },B_{1}^{\xi },...,B_{n-2}^{\xi }\}$ forms the positive
definite Frenet frame along $\xi ,$ let $\kappa _{i},\kappa _{n-1}$ are the
curvatures, the Frenet equations of the null curve $\xi $ are as follows 
\begin{eqnarray*}
T_{\xi }^{\prime } &=&\nabla _{T}\xi ^{\prime }=N_{\xi };B_{1}^{\xi \prime
}=\nabla _{T}N_{\xi }=-\kappa _{1}N_{\xi }+\kappa _{2}B_{2}^{\xi } \\
N_{\xi }^{\prime } &=&\nabla _{T}B_{1}^{\xi }=\kappa _{1}T_{\xi }-B_{1}^{\xi
};B_{2}^{\xi \prime }=-\kappa _{2}T_{\xi }+\kappa _{3}B_{3}^{\xi }
\end{eqnarray*}%
\begin{equation*}
...
\end{equation*}%
\begin{equation}
B_{i}^{\xi \prime }=\nabla _{T}B_{i}^{\xi }=-\kappa _{i}B_{i-1}^{\xi
}+\kappa _{i+1}B_{i+1}^{\xi },i\in \left\{ 3,4,...,n-3\right\}  \tag{3.24a}
\end{equation}%
\begin{equation*}
...
\end{equation*}%
\begin{equation*}
B_{n-2}^{\xi \prime }=\nabla _{T}B_{n-2}^{\xi }=-\kappa _{n-2}B_{n-3}^{\xi },
\end{equation*}%
and%
\begin{equation}
\left\langle T_{\xi },T_{\xi }\right\rangle =\left\langle B_{1}^{\xi
},B_{1}^{\xi }\right\rangle =0,\left\langle T_{\xi },B_{1}^{\xi
}\right\rangle =1;\left\langle N_{\xi },N_{\xi }\right\rangle =\left\langle
B_{j}^{\xi },B_{j}^{\xi }\right\rangle =1,  \tag{3.24b}
\end{equation}%
where $j=2,3,...,n-2$ and $\nabla $ is the Levi Civita connection of $L^{n},$
\cite{14}.

\begin{theorem}
Let $\gamma :I\subset 
\mathbb{R}
\longrightarrow L^{n}$ be a unit-speed null curve having no where vanishing $%
n-1$ curvatures $\kappa _{1},\kappa _{2},...,\kappa _{n-1},$ and let $%
g:I\rightarrow R$ be a nowhere vanishing integrable function with at least $%
(n-2)$-times differentiable primitive function $G$. If $\xi $ is a null $g-$%
rectifying curve in $L^{n}$, then the following statements are satisfied

1) The norm function $l$ associated with null $g-$rectifying curve $\xi
_{g}^{r}$ is explicitly defined as $%
l^{2}=2w_{0}^{r}w_{1}^{r}+c^{2}-w_{1}^{r}{}^{2}$, $c\in 
\mathbb{R}
_{0}^{+}$.

2) The tangential projection of the $g-$position vector field $\xi _{g}^{r}$
onto the tangent vector $T_{\xi }$ is given by $\left\langle \xi
_{g},B_{1}^{\xi }\right\rangle =\kappa _{1}(s+c).$

3) The normal component $\xi _{g}^{rN}(s)$ of the $g-$position vector field $%
\xi _{g}^{r}$ maintains a constant magnitude.

4) The binormal components of the $g-$position vector field $\xi _{g}^{r}$
are, respectively, provided by the following expressions%
\begin{equation*}
\left\langle \xi _{g}^{r},T_{\xi }\right\rangle =s+c;\left\langle \xi
_{g}^{r},B_{i}^{\xi }\right\rangle =\frac{1}{\kappa _{i}}(w_{i-1}^{r\prime
}+\kappa _{i-1}w_{i-2}^{r});\left\langle \xi _{g}^{r},B_{n-2}^{\xi
}\right\rangle =-\int \kappa _{n-2}w_{n-3}^{r}ds.
\end{equation*}

5) The null $g-$rectifying curve of $\xi $ is given as  
\begin{equation*}
\xi _{g}^{r}(s)=\left( s+c,0,...,\frac{1}{\kappa _{i}}(w_{i-1}^{r\prime
}+\kappa _{i-1}w_{i-2}^{r}),...,-\int \kappa _{n-2}w_{n-3}^{r}ds\right) ,
\end{equation*}%
where for $i=2,3,...,n-3.$
\end{theorem}

\begin{proof}
We investigate an null $g-$rectifying curve $\xi :I\subset 
\mathbb{R}
\longrightarrow L^{n}$, defined by its $(n-1)$ non-vanishing curvatures $%
\kappa _{1},\kappa _{2},...,\kappa _{n-1}$. This curve is intrinsically
linked to a nowhere vanishing integrable function $g:I\rightarrow R$, whose
primitive function $G$ is differentiable at least $(n-2)$ times, there exist
differentiable functions $w_{0}^{r},w_{i}^{r}\in C^{\infty }$(for $%
i=0,1,2,...,n-2$ ) such that null $g-$position vector field $\xi _{g}^{r}$
of $\xi $ satisfies equation (3.1). Then, by differentiating of (4.1) and by
using of the Frenet-Serret formulae (2.2), one obtains 
\begin{equation*}
g(s)\overrightarrow{T}=\left( w_{0}^{r}{}^{\prime }-\kappa
_{2}w_{2}^{r}\right) \overrightarrow{T_{\xi }}+(w_{0}^{r}-\kappa
_{1}w_{1}^{r})\overrightarrow{N_{\xi }}+w_{1}^{r\prime }\overrightarrow{%
B_{1}^{\xi }}+(w_{2}^{r\prime }+\kappa _{2}w_{1}^{r}-\kappa _{3}w_{3}^{r})%
\overrightarrow{B_{2}^{\xi }}
\end{equation*}%
\begin{equation}
+\underset{i=3}{\overset{n-3}{\sum }}(w_{i}^{r\prime }+\kappa
_{i}w_{i-1}^{r}-\kappa _{i+1}w_{i+1}^{r})\overrightarrow{B_{i}^{\xi }}%
+(w_{n-2}^{r\prime }+\kappa _{n-2}w_{n-3}^{r})\overrightarrow{B_{n-2}^{\xi }}%
.  \tag{3.25}
\end{equation}

The resulting set of relations is as follows:%
\begin{equation}
g(s)=w_{0}^{r}{}^{\prime }-\kappa _{2}w_{2}^{r}  \tag{3.26a}
\end{equation}%
\begin{equation}
w_{0}^{r}-\kappa _{1}w_{1}^{r}=0  \tag{3.26b}
\end{equation}%
\begin{equation}
w_{1}^{r\prime }=0  \tag{3.26c}
\end{equation}%
\begin{equation}
w_{2}^{r\prime }+\kappa _{2}w_{1}^{r}-\kappa _{3}w_{3}^{r}=0  \tag{3.26d}
\end{equation}%
\begin{equation}
w_{i}^{r\prime }+\kappa _{i}w_{i-1}^{r}-\kappa _{i+1}w_{i+1}^{r}=0 
\tag{3.26e}
\end{equation}%
\begin{equation}
w_{n-2}^{r\prime }+\kappa _{n-2}w_{n-3}^{r}=0.  \tag{3.26f}
\end{equation}

Then, by considering the $n-1$ relations, the following equalities are
written as%
\begin{equation}
w_{0}^{r}=\kappa _{1}(s+c_{1})=\int \left( g(s)+\kappa _{2}w_{2}^{r}\right)
ds  \tag{3.27a}
\end{equation}%
\begin{equation}
w_{1}^{r}=s+c_{1}  \tag{3.27b}
\end{equation}%
\begin{equation}
w_{2}^{r}=(-g(s)+\kappa _{1})\frac{1}{\kappa _{2}}  \tag{3.27c}
\end{equation}%
\begin{equation}
w_{n-2}^{r}=-\int \kappa _{n-2}w_{n-3}^{r}ds  \tag{3.27d}
\end{equation}%
\begin{equation}
w_{i+1}^{r}=\frac{1}{\kappa _{i+1}}(w_{i}^{r\prime }+\kappa
_{i}w_{i-1}^{r}),i=2,3,...,n-3,  \tag{3.27e}
\end{equation}%
by multiplying equations (3.26c), (3.26e), and (3.26f) by $w_{1}^{r}$, $%
w_{i}^{r}$, and $w_{n-2}^{r}$, for $i=2,3,...,n-3$, respectively, and
summing them, one obtains 
\begin{equation}
w_{0}^{r}w_{0}^{r\prime }+\underset{i=2}{\overset{n-3}{\sum }}%
(w_{i}^{r}w_{i}^{r\prime }+\kappa _{i}w_{i-1}^{r}w_{i}^{r}-\kappa
_{i+1}w_{i+1}^{r}w_{i}^{r})+w_{n-2}^{r}w_{n-2}^{r\prime }+\kappa
_{n-2}w_{n-3}^{r}w_{n-2}^{r}=0,  \tag{3.28}
\end{equation}%
from the necessary calculations, the following expression is obtained%
\begin{equation}
\underset{i=1}{\overset{n-2}{\sum }}w_{i}^{r2}=c^{2};c\in 
\mathbb{R}
_{0}.  \tag{3.29}
\end{equation}

Thus, from equations (3.1), (3.24), and (3.29), the norm of the null curve
is calculated as follows%
\begin{equation}
l^{2}=\left\Vert \xi _{g}^{r}\right\Vert
^{2}=2w_{0}^{r}w_{1}^{r}+c^{2}-w_{1}^{r}{}^{2},c\in 
\mathbb{R}
_{0}^{+}.  \tag{3.30}
\end{equation}

Equations (3.1) and (3.5) provide the tangential component equalities for
the null $g$-rectifying curve%
\begin{equation*}
\left\langle \xi _{g}^{r},B_{1}\right\rangle =w_{0}^{r}=\kappa _{1}(s+c_{1}).
\end{equation*}

Furthermore, one can express this as follows using the normal component of
curve $\xi _{g}^{r}$ 
\begin{equation}
\xi _{g}^{r}(s)=w_{0}^{r}T_{\xi }(s)+\xi _{g}^{rN}(s)  \tag{3.31}
\end{equation}%
and the norm of the normal component is computed as 
\begin{equation}
\left\Vert \xi _{g}^{rN}(s)\right\Vert =\sqrt{\overset{n-2}{\underset{i=1}{%
\sum }w_{i}^{r2}}}=c.  \tag{3.32}
\end{equation}

It is thus confirmed that the norm of the normal component is constant,
which solidifies the proof of (3). Moving forward, through the consideration
of equations (3.1) and (3.5), the binormal components are given as presented
below%
\begin{equation}
\left\langle \xi _{g}^{r},T_{\xi }\right\rangle =w_{1}^{r}=s+c;\left\langle
\xi _{g}^{r},B_{i}^{\xi }\right\rangle =w_{i}^{r}=\frac{1}{\kappa _{i}}%
(w_{i-1}^{r\prime }+\kappa _{i-1}w_{i-2}^{r}),i=2,3,...,n-3  \tag{3.33}
\end{equation}%
\begin{equation}
\left\langle \xi _{g}^{r},B_{n-2}^{\xi }\right\rangle =w_{n-2}^{r}=-\int
\kappa _{n-2}w_{n-3}^{r}ds,  \tag{3.34}
\end{equation}%
and this proves statement (4).

Conversely, the validity of statement (1) or (2) strictly implies that for a
unit-speed $g-$rectifying null curve $\xi _{g}^{r}$ in $L^{n}$, possessing
nowhere vanishing ($n-1$) curvatures $\kappa _{i}$, and associated with a
nowhere vanishing integrable function $g$ (having an ($n-2$)-times
differentiable primitive $G$) 
\begin{equation*}
\left\langle \xi _{g}^{r},B_{1}^{\xi }\right\rangle =w_{0}^{r}=\kappa
_{1}(s+c),
\end{equation*}%
by differentiating the last equation and from (2.24), by using (3.26a) and
since $\left\langle T_{\xi },B_{1}^{\xi }\right\rangle =1$, the following
equality is obtained%
\begin{equation*}
\left\langle \xi _{g}^{r},N_{\xi }\right\rangle =0,
\end{equation*}%
it is therefore shown that $\xi _{g}^{r}$ belongs to the rectifying space of 
$\xi $, unequivocally identifying $\xi $ as a null $g-$rectifying curve in $%
L^{n}.$ Under the premise that statement (3) is true, the normal component $%
\xi _{g}^{rN}$ maintains a constant value, as revealed by the relation $\xi
_{g}^{r}(s)=w_{0}^{r}T_{\xi }(s)+\xi _{g}^{rN}(s)$. Thus, the norm is
written as follows%
\begin{equation*}
\left\Vert \xi _{g}^{r}\right\Vert =\sqrt{%
2w_{0}^{r}w_{1}^{r}+c^{2}-w_{1}^{r}{}^{2}}.
\end{equation*}

Through the differentiation of the preceding relation, coupled with the
Frenet-Serret formulae (3.24), it is determined that $\left\langle \xi
_{g}^{r},N_{\xi }\right\rangle =0$, establishing $\xi $ as an null $g-$%
rectifying curve in $L^{n}$. Subsequently, under the premise of statement
(4)'s validity, one gets%
\begin{equation*}
\left\langle \xi _{g}^{r},B_{i}^{\xi }\right\rangle =w_{i}^{r}=\frac{1}{%
\kappa _{i}}(w_{i-1}^{r\prime }+\kappa _{i-1}w_{i-2}^{r}),i=2,3,...,n-3.
\end{equation*}

By taking the derivative of the previous equations and utilizing (3.24), one
can establish, based on (3.26) and (3.27), that $\left\langle \xi
_{g}^{r},N_{\xi }(s)\right\rangle =0$. Consequently, this observation
rigorously confirms that the curve is indeed a $g-$rectifying null curve in $%
L^{n}$. Finally, it is clear that statement (5) holds.
\end{proof}

\section{Representation of $g-$normal curves in Lorentzian $n-$space $L^{n}$}

A detailed exploration of unit-speed $g-$normal curves in $L^{n}$ is
undertaken in this section, focusing on the characterizations derived from
an analysis of the tangential, normal, and binormal constituents of their $%
g- $position vector fields.

\begin{definition}
Let $\xi :I\subset 
\mathbb{R}
\longrightarrow L^{n}$ is called normal curve if for all $s\in I$, the
orthogonal complement of $T$ contains a fixed point. Then, for $T^{\bot }$
the orthogonal complement of $T$, the position vector of a spacelike normal
curve $\xi $ in $L^{n}$ can be written as%
\begin{equation}
\xi _{g}^{n}(s)=\vartheta N_{\xi }(s)+\overset{n-2}{\underset{i=1}{\sum }}%
\mu _{i}B_{i}^{\xi }(s);\mu _{i}\in C^{\infty },i=1,...,n-2.  \tag{4.1}
\end{equation}
\end{definition}

\begin{theorem}
Let $\xi :I\subset 
\mathbb{R}
\longrightarrow L^{n}$ be a unit-speed curve having no where vanishing $n-1$
curvatures $\kappa _{1},\kappa _{2},...,\kappa _{n-1},$ and let $%
g:I\rightarrow R$ be a no where vanishing integrable function with at least $%
(n-2)$-times differentiable primitive function $G$. If $\xi $ is a spacelike 
$g-$normal curve in $L^{n}$, then the following statements are satisfied

1) The norm function $l$ associated with the $g-$position vector field $\xi
_{g}^{n}$ is explicitly defined as $l^{2}=-\left( \frac{g(s)}{\kappa _{1}}%
\right) ^{2}+c^{2}$, where $G(s)$ represents the primitive function and $c$
is a specified non-zero constant.

2) The normal projection of the $g-$position vector field $\xi _{g}^{n}$
onto the tangent vector $N_{\xi }$ is given by the scalar product $%
\left\langle \xi _{g}^{n},N_{\xi }\right\rangle =\frac{g(s)}{\kappa _{1}}.$

3) The binormal component $\xi _{g}^{nB}(s)$ of the $g-$position vector
field $\xi _{g}^{n}$ maintains a constant magnitude.

4) The binormal components of the $g-$position vector field $\xi _{g}^{n}$
are, respectively, provided by the following expressions%
\begin{equation*}
\left\langle \xi _{g}^{n},B_{1}^{\xi }\right\rangle =\frac{-1}{\varepsilon
_{1}\kappa _{2}}\left( \frac{g(s)}{\kappa _{1}}\right) ^{\prime
};\left\langle \xi _{g}^{n},B_{n-2}^{\xi }\right\rangle =-\varepsilon
_{n-1}\int \kappa _{n-1}\mu _{n-3}ds,
\end{equation*}%
\begin{equation*}
\left\langle \xi _{g}^{n},B_{2}^{\xi }\right\rangle =\frac{-1}{\varepsilon
_{1}\varepsilon _{2}\kappa _{3}}\left( \frac{\kappa _{2}}{\kappa _{1}}%
g(s)+\left( \frac{1}{\varepsilon _{2}\kappa _{2}}\left( \frac{g(s)}{\kappa
_{1}}\right) ^{\prime }\right) ^{\prime }\right) 
\end{equation*}%
\begin{equation*}
\left\langle \xi _{g}^{n},B_{i+1}^{\xi }\right\rangle =\frac{1}{\varepsilon
_{i+1}\kappa _{i+2}}(\mu _{i}^{\prime }+\kappa _{i+1}\mu
_{i-1}),i=2,3,...,n-3.
\end{equation*}

5) The spacelike $g-$normal curve of $\xi $ is given as 
\begin{equation*}
\xi _{g}^{n}(s)=\left( 
\begin{array}{c}
0,\frac{g(s)}{\kappa _{1}},\frac{-1}{\varepsilon _{1}\kappa _{2}}\left( 
\frac{g(s)}{\kappa _{1}}\right) ^{\prime },\frac{-1}{\varepsilon
_{1}\varepsilon _{2}\kappa _{3}}\left( \frac{\kappa _{2}}{\kappa _{1}}%
g(s)+\left( \frac{1}{\varepsilon _{2}\kappa _{2}}\left( \frac{g(s)}{\kappa
_{1}}\right) ^{\prime }\right) ^{\prime }\right)  \\ 
,...,\frac{1}{\varepsilon _{i+1}\kappa _{i+2}}(\mu _{i}^{\prime }+\kappa
_{i+1}\mu _{i-1}),...,-\varepsilon _{n-1}\int \kappa _{n-1}\mu _{n-3}ds%
\end{array}%
\right) .
\end{equation*}
\end{theorem}

\begin{proof}
For an spacelike $g-$normal curve $\xi :I\subset 
\mathbb{R}
\longrightarrow L^{n}$, defined by its $(n-1)$ nowhere vanishing curvatures $%
\kappa _{i}$, and associated with a nowhere vanishing integrable function $%
g:I\rightarrow R$ (with an ($n-2$)-times differentiable primitive $G$),
there exist differentiable functions $\vartheta ,\eta _{i}\in C^{\infty }$
(for $i=1,2,...,n-2$) such that its $g-$position vector field $\xi _{g}^{n}$
satisfies equation (3.1). Subsequent differentiation of (4.1) and the
application of the Frenet-Serret formulae (2.2) then yield 
\begin{equation*}
g(s)\overrightarrow{T}=(-\varepsilon _{1}\kappa _{1}\vartheta )%
\overrightarrow{T_{\xi }}+\left( \vartheta ^{\prime }-\varepsilon
_{1}\varepsilon _{2}\kappa _{2}\mu _{1}\right) \overrightarrow{N_{\xi }}%
+(\kappa _{2}\vartheta +\mu _{1}^{\prime }-\varepsilon _{2}\varepsilon
_{3}\kappa _{3}\mu _{2})\overrightarrow{B_{1}^{\xi }}
\end{equation*}%
\begin{equation}
+\underset{i=2}{\overset{n-3}{\sum }}(\mu _{i}^{\prime }+\kappa _{i+1}\mu
_{i-1}-\varepsilon _{i+1}\varepsilon _{i+2}\kappa _{i+2}\mu _{i+1})%
\overrightarrow{B_{i}^{\xi }}+(\mu _{n-2}^{\prime }+\kappa _{n-1}\mu _{n-3})%
\overrightarrow{B_{n-2}^{\xi }}.  \tag{4.2}
\end{equation}

Hence, the following relationships are established%
\begin{equation}
g(s)=-\varepsilon _{1}\kappa _{1}\vartheta  \tag{4.3a}
\end{equation}%
\begin{equation}
\vartheta ^{\prime }-\varepsilon _{1}\varepsilon _{2}\kappa _{2}\mu _{1}=0 
\tag{4.3b}
\end{equation}%
\begin{equation}
\kappa _{2}\vartheta +\mu _{1}^{\prime }-\varepsilon _{2}\varepsilon
_{3}\kappa _{3}\mu _{2}=0  \tag{4.3c}
\end{equation}%
\begin{equation}
\mu _{i}^{\prime }+\kappa _{i+1}\mu _{i-1}-\varepsilon _{i+1}\varepsilon
_{i+2}\kappa _{i+2}\mu _{i+1}=0  \tag{4.3d}
\end{equation}%
\begin{equation}
\mu _{n-2}^{\prime }+\kappa _{n-1}\mu _{n-3}=0.  \tag{4.3e}
\end{equation}

From the $(n-1)$ relations within the stated system of equations, the
following equalities are deduced%
\begin{equation}
\vartheta =\frac{-g(s)}{\varepsilon _{1}\kappa _{1}}  \tag{4.4a}
\end{equation}%
\begin{equation}
\mu _{1}=\frac{1}{\varepsilon _{1}\varepsilon _{2}\kappa _{2}}\left( \frac{%
-g(s)}{\kappa _{1}}\right) ^{\prime }  \tag{4.4b}
\end{equation}%
\begin{equation}
\mu _{2}=\frac{1}{\varepsilon _{1}\varepsilon _{2}\varepsilon _{3}\kappa _{3}%
}(-\frac{\kappa _{2}}{\kappa _{1}}g(s)+\left( \frac{1}{\varepsilon
_{2}\kappa _{2}}\left( \frac{-g(s)}{\kappa _{1}}\right) ^{\prime }\right)
^{\prime })  \tag{4.4c}
\end{equation}%
\begin{equation}
\mu _{n-2}=-\int \kappa _{n-1}\mu _{n-3}ds  \tag{4.4d}
\end{equation}%
\begin{equation}
\mu _{i+1}=\frac{1}{\varepsilon _{i+1}\varepsilon _{i+2}\kappa _{i+2}}(\mu
_{i}^{\prime }+\kappa _{i+1}\mu _{i-1}),i=2,3,...,n-3.  \tag{4.4e}
\end{equation}

The linear combination of equations (4.3c), (4.3d), and (4.3e), employing $%
\mu _{1}$, $\mu _{i}$ ($i=2,3,...,n-3$), and $\mu _{n-2}$ as respective
multipliers, culminates in the following expression.%
\begin{equation}
\underset{i=2}{\overset{n-3}{\sum }}\varepsilon _{i+1}\mu
_{i}^{2}=c^{2};c\in 
\mathbb{R}
_{0}^{+}.  \tag{4.5}
\end{equation}

The calculation of the $g-$normal curve's norm, utilizing equations (4.1)
and (2.2), proceeds as follows%
\begin{equation}
l^{2}=\left\Vert \xi _{g}^{n}\right\Vert ^{2}=\varepsilon _{1}\vartheta
^{2}+c^{2}=-\left( \frac{g(s)}{\kappa _{1}}\right) ^{2}+c^{2}.  \tag{4.6}
\end{equation}

The normal component of the $g-$normal curve in question is found by
utilizing equation (4.1), which yields the following equality. 
\begin{equation*}
\left\langle \xi _{g}^{n},N_{\xi }\right\rangle =\varepsilon _{1}\vartheta =%
\frac{g(s)}{\kappa _{1}}.
\end{equation*}

In addition, this can be formulated as shown below, utilizing the normal
component of curve $\xi _{g}^{n}$ 
\begin{equation}
\xi _{g}^{n}(s)=\vartheta N_{\xi }(s)+\overset{n-2}{\underset{i=1}{\sum }}%
\mu _{i}B_{i}^{\xi }(s)=\vartheta N_{\xi }(s)+\xi _{g}^{nB}(s)  \tag{4.7}
\end{equation}%
and the norm of the binormal component is given as follows%
\begin{equation}
\left\Vert \xi _{g}^{nB}(s)\right\Vert =\sqrt{\overset{n-2}{\underset{i=1}{%
\sum }}\mu _{i}^{2}\varepsilon _{i+1}}=c,  \tag{4.8}
\end{equation}%
the constant nature of the normal component's norm is thus affirmed, which
solidifies the proof of (3). Proceeding, and considering equations (4.1),
the binormal components are ascertained as presented below%
\begin{equation}
\left\langle \xi _{g}^{n},B_{1}^{\xi }\right\rangle =\mu _{1}\varepsilon
_{2}=\frac{1}{\varepsilon _{1}\kappa _{2}}\left( \frac{-g(s)}{\kappa _{1}}%
\right) ^{\prime }  \tag{4.9}
\end{equation}%
\begin{equation}
\left\langle \xi _{g}^{n},B_{2}^{\xi }\right\rangle =\mu _{2}\varepsilon
_{3}=\frac{1}{\varepsilon _{1}\varepsilon _{2}\kappa _{3}}(-\frac{\kappa _{2}%
}{\kappa _{1}}g(s)+\left( \frac{1}{\varepsilon _{2}\kappa _{2}}\left( \frac{%
-g(s)}{\kappa _{1}}\right) ^{\prime }\right) ^{\prime })  \tag{4.10}
\end{equation}%
\begin{equation}
\left\langle \xi _{g}^{n},B_{i+1}^{\xi }\right\rangle =\mu _{i+1}\varepsilon
_{i+2}=\frac{1}{\varepsilon _{i+1}\kappa _{i+2}}(\mu _{i}^{\prime }+\kappa
_{i+1}\mu _{i-1}),i=2,3,...,n-3  \tag{4.11}
\end{equation}%
\begin{equation}
\left\langle \xi _{g}^{n},B_{n-2}^{\xi }\right\rangle =\mu _{n-2}\varepsilon
_{n-1}=-\varepsilon _{n-1}\int \kappa _{n-1}\mu _{n-3}ds,  \tag{4.12}
\end{equation}%
thus the statement (4) is proved.

Consider, conversely, a unit-speed $g-$normal curve $\xi _{g}^{n}$,
possessing $(n-1)$ nowhere vanishing curvatures $\kappa _{1},\kappa
_{2},...,\kappa _{n-1}$, and a nowhere vanishing integrable function $g$.
Should either statement (1) or (2) be true, it then demonstrably follows
that 
\begin{equation*}
\left\langle \xi _{g}^{n},N_{\xi }\right\rangle =\frac{g(s)}{\kappa _{1}},
\end{equation*}%
by differentiating the previous equality and considering equations (2.2),
one gets%
\begin{equation*}
\left\langle \xi _{g}^{n},T_{\xi }\right\rangle =0,
\end{equation*}%
the fact that $\xi _{g}^{n}$ lies in the normal space of $\xi $ ultimately
serves as proof that $\xi $ is an $g-$normal curve in $L^{n}.$

The truth of statement (3) implies that the binormal component $\xi
_{g}^{nB} $ is a constant, which we denote by $c$. From this, and utilizing
equation (4.7), the norm value is derived as follows%
\begin{equation*}
\left\Vert \xi _{g}^{nB}\right\Vert ^{2}=\overset{n-2}{\underset{i=1}{\sum }}%
\mu _{i}^{2}\varepsilon _{i+1}=c^{2}.
\end{equation*}

Differentiating the preceding equation and employing the Frenet-Serret
formulae (2.2) yields $\left\langle \xi _{g}^{n},T_{\xi }\right\rangle =0$.
This result confirms that $\xi $ constitutes a $g-$normal curve in $L^{n}.$

If statement (4) is taken as true, the first and second binormal components
of $\xi _{g}^{n}$ are given by equations (4.9)-(4.12). Applying (2.2) to the
derivatives of these equations leads to $\left\langle \xi _{g}^{n},T_{\xi
}\right\rangle =0$, which conclusively establishes the curve as a spacelike $%
g-$normal curve in $L^{n},$ The validity of (5) is readily apparent.
\end{proof}

\section{Conclusion}

In this study, we introduced and thoroughly analyzed the generalized
concepts of the spacelike $g-$rectifying curves, the null $g-$rectifying
curves and the spacelike $n-$normal curves within the intricate frame of
Lorentzian $n-$space. Building upon the foundational definitions of
classical rectifying and normal curves, our novel approach incorporated an $%
g-$position vector field, defined as (2.5), where $g$ represents a nowhere
vanishing integrable function. This generalization allowed for a more
flexible and comprehensive characterization of curve geometry, particularly
considering the unique properties of both spacelike and null curves in this
indefinite metric environment. Our primary objective of providing a
comprehensive characterization and classification of these spacelike $g-$%
rectifying(or null $g-$rectifying) and $g-$normal curves has been achieved
through rigorous mathematical treatment. We have elucidated the conditions
under which these generalized curves exist and explored their fundamental
properties, thereby significantly expanding the existing understanding of
curves in Lorentzian $n-$spaces. Prospective research avenues include
extending these generalized concepts to other geometries (pseudo-Galilean $%
n- $spaces, Lightlike cone $n-$space,...), investigating higher-order
generalizations of position vector fields. This research underscores the
continuous evolution of curve theory and its profound implications for
understanding geometric structures in complex mathematical and physical
domains.

\section*{Funding}

Not applicable.

\section*{Informed Consent Statement}

Not applicable.

\section*{Conflicts of Interest}

The author declares no conflict of interest.

\end{document}